\documentclass[dvipdfmx]{amsart}
\usepackage{amsmath,amssymb}
\usepackage{amsthm}
\usepackage{amsfonts}
\usepackage{mathrsfs}
\usepackage{mathtools}

\theoremstyle{definition}
\newtheorem{thm}{Theorem}[section]
\newtheorem{prop}[thm]{Proposition}
\newtheorem{lem}[thm]{Lemma}
\newtheorem{cor}[thm]{Corollary}
\newtheorem{dfn}[thm]{Definition}
\newtheorem{rem}[thm]{Remark}
\newtheorem{prob}[thm]{Problem}
\newtheorem{exam}[thm]{Example}

\usepackage{hyperref}
\subjclass[2020]{14B10, 14E15, 14M25}

\newcommand{\C}{\mathbb{C}}
\newcommand{\R}{\mathbb{R}}
\newcommand{\Z}{\mathbb{Z}}
\newcommand{\Q}{\mathbb{Q}}
\newcommand{\A}{\mathbb{A}}
\newcommand{\Spec}{\operatorname{Spec}}
\newcommand{\FB}{\operatorname{FB}}
\newcommand{\Hilb}{\operatorname{Hilb}}
\newcommand{\Cone}{\operatorname{Cone}}
\newcommand{\Conv}{\operatorname{Conv}}
\newcommand{\mult}{\operatorname{mult}}
\newcommand{\Gen}{\operatorname{Gen}}
\newcommand{\GL}{\operatorname{GL}}
\newcommand{\diag}{\operatorname{diag}}

\begin{document}
\title{Properties of moderate toric resolutions in dimension three}
\author{Yutaro Kaijima and Yudai Yamamoto}
\address{Department of Mathematics, Graduate School of Science, Osaka University Toyonaka, Osaka 560-0043, JAPAN}
\email{u797381f@ecs.osaka-u.ac.jp}
\address{Department of Mathematics, Graduate School of Science, Osaka University Toyonaka, Osaka 560-0043, JAPAN}
\email{yudai.yamamoto.math@gmail.com}

\begin{abstract}
We study moderate toric resolutions introduced by Ch\'avez-Mart\'inez, Duarte and Yasuda, which appears in the relation between F-blowups and essential divisors. In particular, we address the problems, when it exists, and if it is the case, what properties it has in conjunction with the birational geometry and Hilbert basis resolutions, mainly in dimension three.
\end{abstract}

\maketitle

\section{Introduction}

For a non-negative integer $e$, the $e$-th F-blowup $\FB_e(X)$ of a variety $X$ in positive characteristic, introduced by Yasuda \cite{Yas12}, defined to be the universal birational flattening of the $e$-th  iterate of the Frobenius morphism. A natural problem on the F-blowups is when $\FB_e(X)$ become a (crepant) resolution for $e\gg0$ ? There are several known results on this problem, either positive or negative, obtained for example in \cite{Yas12,TY09,Har12,Har15,HS11,HSY13,LY24}. Since the answer is not always positive, we may consider the following problem as a next step.
\begin{prob}\label{essential divisor}
For a normal variety $X$ and $e\gg0$, do all essential divisors over $X$ appear on $\FB_e(X)$ as prime divisors ?
\end{prob}
Here, an \emph{essential divisor} over $X$ means a prime divisor which appears on every resolution of $X$. The notion of essential divisors is famous for its appearance in the Nash problem on families of arcs (for example, see \cite{PS15}). Since each point of the F-blowups can be regarded as a jet \cite[Introduction]{Yas12}, Problem \ref{essential divisor} can be viewed as an analogue of the Nash problem.

In this paper, we consider Problem \ref{essential divisor} in the case of normal affine toric varieties. In this case, the Frobenius-like morphism can be defined even when the charactersitic of the base field is zero. Accordingly, we can define F-blowups of normal affine toric varieties in arbitrary characteristic.
\begin{dfn}[\cite{Yas12}]
Let $M$ and $N$ be free abelian groups of rank $d$ which are dual to each other and let $\sigma\subseteq N_\R\coloneqq N\otimes\R$ be a strongly convex, rational polyhedral cone. Let $X\coloneqq\Spec k[\sigma^\vee\cap M]$ be the associated toric variety. For $l\in\Z_{>0}$, we define the \emph{$l$-th F-blowup} $\FB_{(l)}(X)$ that the universal birational flattening of the morphism
\begin{equation*}
F_{(l)}\colon\Spec k[\sigma^\vee\cap(1/l)M]\to X
\end{equation*}
which corresponds to the inclusion $\sigma^\vee\cap M\hookrightarrow\sigma^\vee\cap(1/l)M$.
\end{dfn}
In \cite{Yas12}, it was shown that the sequence of F-blowups $(\FB_{(l)}(X))_{l\in\Z_{>0}}$ as well as the sequence of the normalizations of F-blowups $(\widetilde{\FB}_{(l)}(X))_{l\in\Z_{>0}}$ stabilizes. Thus, we write $\FB_{(l)}(X)$ and the normalization $\widetilde{\FB}_{(l)}(X)$ of $\FB_{(l)}(X)$ for sufficiently large integer $l$ as $\FB_{(\infty)}(X)$ and $\widetilde{\FB}_{(\infty)}(X)$, respectively. Note that if the characteristic of $k$ is $p>0$, then $\FB_{(p^e)}(X)\cong\FB_e(X)$.

In this paper, we mainly consider the following version of essential divisors due to Bouvier and Gonzalez-Sprinberg.
\begin{dfn}[\cite{BGS95}]
We say that a torus invariant divisor over a normal affine toric variety $X$ is a \emph{BGS essential divisor} if it appears on every toric resolution of $X$ as a prime divisor.
\end{dfn}
One of the properties of BGS essential divisors is that there is a one-to-one correspondence between the elements of the Hilbert basis of the cone and the BGS essential divisors over the normal affine toric variety associated to the cone \cite[Theorem 1.10]{BGS95}. Recently, Ch\'avez-Mart\'inez, Duarte and Yasuda gave an affirmative answer to Problem \ref{essential divisor} by using the following notion.
\begin{dfn}[\cite{CMDY24}]\label{moderate toric resolution intro}
Let $\Sigma$ be a subdivision of strongly convex, nondegenerate, rational polyhedral cone $\sigma\subseteq N_\R$. We say that $\Sigma$ is a \emph{moderate toric resolution} of $\sigma$ if every nondegenerate cone $\tau$ of $\Sigma$ is smooth and the affine hyperplane spanned by the minimal generators of $\tau$ intersects with every one-dimensional face of $\sigma$. We also call the induced morphism of schemes $X_\Sigma\to X_\sigma$ a moderate toric resolution.
\end{dfn}
\begin{thm}[{\cite[Corollary 6.9]{CMDY24}}]
Let $X$ be an affine toric variety associated to a cone $\sigma\subseteq N_\R$ which is a strongly convex, rational polyhedral and nondegenerate. If $X$ admits a moderate toric resolution, then every BGS essential divisor over $X$ appears on $\widetilde{\FB}_{(\infty)}(X)$ as a prime divisor.
\end{thm}
In this paper, we study the properties of moderate toric resolutions (Definition \ref{moderate toric resolution intro}). In particular, we ask when a moderate toric resolution exists, and if so, what properties it has and how it can be characterized. It is already known that a toric crepant resolution of arbitrary dimension and a minimal resolution of toric surfaces are moderate toric resolutions \cite{CMDY24}.

The following two theorems are our main results:
\begin{thm}[Theorem \ref{d dim no moderate}]\label{d dim no moderate intro}
Let $d\geq3$ and let $\sigma\subseteq N_\R$ be a $d$-dimensional cone generated by
\begin{equation*}
(1,0,\ldots,0),(0,1,0,\ldots,0),\ldots,(0,\ldots,0,1,0),(a_1,a_2,\ldots,a_d)
\end{equation*}
where $a_1,a_2\ldots,a_d\in\Z$, $a_1\leq a_2\leq\cdots\leq a_{d-1}<a_d$ and $1\leq a_1\leq d-2$. Then the affine toric variety associated to $\sigma$ does not admit a moderate toric resolution.
\end{thm}
\begin{thm}[Theorem \ref{moderate is minQterminal}]\label{moderate is minQterminal intro}
Let $X$ be a three-dimensional affine toric variety. If $X$ admits a moderate toric resolution, then it is a minimal terminal $\Q$-factorial model of $X$.
\end{thm}
The following corollaries follow from these theorems.
\begin{cor}[Corollary \ref{3Qterminal no moderate}, Corollary \ref{4GorQterminal no moderate}]\label{3Qterminal no moderate intro}
Let $X$ be a three-dimensional terminal $\Q$-factorial singular affine toric variety or four-dimensional Gorenstein terminal $\Q$-factorial singular affine toric variety. Then $X$ does not admit a moderate toric resolution.
\end{cor}
\begin{cor}[Corollary \ref{moderate is crepant}]\label{moderate is crepant intro}
Let $X$ be a three-dimensional canonical affine toric variety. If $X$ admits a moderate toric resolution, then it is a crepant resolution of $X$.
\end{cor}
The proofs of the above theorems and corollaries are outlined as follows. Theorem \ref{d dim no moderate intro} is proved by focusing on the nondegenerate cone that always appears in a moderate toric resolution, and leads to a contradiction. Moreover, using the terminal lemma and results from \cite{MS84}, we see that Corollary \ref{3Qterminal no moderate intro} is a special case of Theorem \ref{d dim no moderate intro}. Theorem \ref{moderate is minQterminal intro} is reduced to the case of terminal $\Q$-factorial affine toric varieties and is shown by using Corollary \ref{3Qterminal no moderate intro}. Colollary \ref{moderate is crepant intro} is a special case of Theorem \ref{moderate is minQterminal intro}. In the last section, as an analogue of Corollary \ref{3Qterminal no moderate intro}, we give an example of a four-dimensional Gorenstein terminal $\Q$-factiorial affine toric variety which does not admit a Hilbert basis resolution. It is a generalization of \cite[Example 3.1]{BGS95}.

\section*{Acknowledgements}

We would like to thank our supervisor Takehiko Yasuda, who gave constructive comments to improve this paper. We would also like to thank Daniel Duarte, Yusuke Sato and Kohei Sato for valuable comments and suggestions. The second named author was supported by JST SPRING, Grant Number JPMJSP2138.

\section{Preliminalies}

Let $M$ and $N$ be free abelian groups of rank $d$ which are dual to each other, and let $N_\R\coloneqq N\otimes\R$ and $M_\R\coloneqq M\otimes\R$, and let $k$ be an arbitrary field. By a \emph{cone}, we mean a strongly convex, rational polyhedral cone. A \emph{primitive element} of $N$ means a element $n\in N$ such that $\varepsilon n\notin N$ for each $0<\varepsilon<1$. The basic terminology of toric variety conforms to \cite{CLS11}. In particular, we use the following notation for cones and convex hulls:
\begin{align*}
&\Cone(u_1,\ldots,u_s)\coloneqq\left\{\sum_{i=1}^s\lambda_iu_i\mid \lambda_i\geq0\right\}, \\
&\Conv(S)\coloneqq\left\{\sum_{u\in S}\lambda_uu\mid\lambda_u\geq0,\sum_{u\in S}\lambda_u=1\ \text{and}\ \lambda_u=0\ \text{for all but finitely many}\ u\right\},
\end{align*}
where $u_1,\ldots,u_s\in N$ are primitive elements and $S\subseteq N_\R$ is a set.

Let $\sigma\subseteq N_\R$ be a cone and let $\sigma^\vee\subseteq M_\R$ be the dual cone of $\sigma$. The affine toric variety associated to $\sigma$ is given by
\begin{equation*}
X_\sigma\coloneqq\Spec k[\sigma^\vee\cap M].
\end{equation*}
More generally, for a fan $\Sigma$ of $N_\R$, we write the toric variety associated to $\Sigma$ as $X_\Sigma$.

\subsection{BGS essential divisors and Hilbert basis resolutions}

A \emph{torus invariant divisor over} a normal toric variety $X$ means a prime divisor $E$ on $Z$ for some toric proper birational morphism $Z\to X$ of normal toric varieties. We identify two torus invariant divisors over $X$, $E\subset Z$ and $E'\subset Z'$ if $E$ and $E'$ map to each other by the natural birational map between $Z$ and $Z'$.

Let $\sigma\subseteq N_\R$ be a cone, as is well-known, there is a one-to-one correspondence between the set of primitive elements of $N$ contained in $\sigma$ and the set of torus invariant divisors over $X_\sigma$ \cite[pp.53--54]{Ful93}. By this correspondence, one-dimensional faces of $\sigma$ correspond to torus invariant divisors \emph{on} $X_\sigma$.

There are several versions of essential divisors. In this paper, we mainly consider the following version of essential divisors.
\begin{dfn}[BGS essential divisor \cite{BGS95}]
We say that a torus invariant divisor over a normal affine toric variety $X$ is a \emph{BGS essential divisor} if it appears on every toric resolution of $X$ as a prime divisor.
\end{dfn}
\begin{dfn}
For a cone $\sigma\subseteq N_\R$, the \emph{Hilbert basis} of $\sigma$ is
\begin{equation*}
\Hilb_N(\sigma)\coloneqq\{n\in\sigma\cap N\setminus\{0\}\mid\forall n_1,n_2\in\sigma\cap N,n=n_1+n_2\Rightarrow n_1=0\ \text{or}\ n_2=0\}.
\end{equation*}
\end{dfn}
As is well-known, the Hilbert basis is also characterised as the minimal generating set of the monoid $\sigma\cap N$.

Let $\Sigma$ be a fan of $N_\R$ and let $\Sigma(r)$ be the set of the $r$-dimensional cone of $\Sigma$ for each $r$. For $\rho\in\Sigma(1)$, there exists a unique primitive element $P(\rho)\in\rho\cap N$ with $\rho=\R_{\geq0}\cdot P(\rho)$. Therefore, We have the set of minimal generators of $\sigma\in\Sigma$
\begin{equation*}
\Gen(\sigma)\coloneqq\{P(\rho)\mid\rho\in\Sigma(1),\rho\preceq\sigma\},
\end{equation*}
where the symbol $\rho\preceq\sigma$ means that $\rho$ is a face of $\sigma$. We define $\Gen(\Sigma)$ by
\begin{equation*}
\Gen(\Sigma)\coloneqq\bigcup_{\sigma\in\Sigma}\Gen(\sigma).
\end{equation*}
\begin{prop}[{\cite[Theorem 1.10]{BGS95}}]
Let $\sigma\subseteq N_\R$ be a cone. The primitive element $n$ of $\sigma\cap N$ is contained in $\Hilb_N(\sigma)$ if and only if $n$ is contained in $\Gen(\Sigma)$ for every toric resolution $\Sigma$ of $\sigma$. If this is the case, the torus invariant divisor correspond to $n$ is a BGS essential divisor.
\end{prop}
\begin{dfn}[\cite{SS23}]
We say a subdivision $\Sigma$ of a cone $\sigma\subseteq N_\R$ is a \emph{Hilbert basis resolution}\footnote{In \cite{BGS95} and \cite{Dai02}, it is called $G$-desingularization or Hilb-desingularization, but here we will follow \cite{SS23} and call it Hilbert basis resolution.} if $\Sigma$ is smooth and $\Gen(\Sigma)=\Hilb_N(\sigma)$. We also call the induced morphism of schemes $X_\Sigma\to X_\sigma$ a Hilbert basis resolution.
\end{dfn}

\subsection{Moderate toric resolutions}

\begin{dfn}\label{linear function}
Let $\sigma=\Cone(u_1,\ldots,u_d)\subseteq N_\R$ be a nondegenerate simplicial cone. We define $l_\sigma$ as the linear function $N_\R\to\R$ that sends each $u_i$ to 1.
\end{dfn}
Let $\Sigma$ be a resolution of a nondegenerate cone $\sigma$. Then, each nondegenerate cone $\tau\in\Sigma$ is smooth, in particular, simplicial. Therefore, we can consider the linear function $l_\tau$ defined in Definition \ref{linear function}.
\begin{dfn}[\cite{CMDY24}]\label{moderate toric resolution}
Let $\Sigma$ be a subdivision of a nondegenerate cone $\sigma\subseteq N_\R$. We say that $\Sigma$ is a \emph{moderate toric resolution} of $\sigma$ if every nondegenerate cone $\tau$ of $\Sigma$ is smooth and the affine hyperplane $\{x\in N_\R\mid l_\tau(x)=1\}$ spanned by $\Gen(\tau)$ intersects with every one-dimensional face of $\sigma$. We also call the induced morphism of schemes $X_\Sigma\to X_\sigma$ a moderate toric resolution.
\end{dfn} 
\begin{exam}
A toric crepant resolution of arbitrary dimension and a minimal resolution of toric surfaces are moderate toric resolutions \cite{CMDY24}.
\end{exam}
\begin{prop}[{\cite[Proposition 6.7]{CMDY24}}]\label{moderate is Hilbert}
Let $\Sigma$ be a moderate toric resolution of a nondegenerate cone $\sigma\subseteq N_\R$. Then $\Gen(\Sigma)=\Hilb_N(\sigma)$, that is, a moderate toric resolution is a Hilbert basis resolution.
\end{prop}

\subsection{Toric singularities}

\begin{dfn}\label{multiplicity}
Let $\sigma=\Cone(u_1,\ldots,u_s)\subseteq N_\R$ be a simplicial cone. Then the \emph{multiplicity} (\emph{index}) of $\sigma$ is the index
\begin{equation*}
\mult(\sigma)\coloneqq[\sigma\cap N+(-\sigma)\cap N:\Z u_1+\cdots+\Z u_s].
\end{equation*}
\end{dfn}
\begin{prop}[{\cite[Proposition 11.1.8]{CLS11}}]\label{property of multiplicity}
Let $\sigma=\Cone(u_1,\ldots,u_s)\subseteq N_\R$ be a simplicial cone. Then
\begin{enumerate}
\item $\sigma$ is smooth if and only if $\mult(\sigma)=1$.
\item $\mult(\sigma)$ is the number of points in $P_\sigma$, where
\begin{equation*}
P_\sigma=\left\{n\in N\mid n=\sum_{i=1}^s\lambda_iu_i,0\leq\lambda_i<1\right\}.
\end{equation*}
\item Let $e_1,\ldots,e_s$ be a basis of $\sigma\cap N+(-\sigma)\cap N$ and write $u_i=\sum_{j=1}^sa_{ij}e_j$. Then
\begin{equation*}
\mult(\sigma)=|\det(a_{ij})|.
\end{equation*}
\end{enumerate}
\end{prop}
\begin{rem}\label{canonical terminal}
The affine toric variety $X_\sigma$ has a canonical model $X_{\Sigma_\text{c}}$, associated to the subdivision $\Sigma_\text{c}$ of $\sigma$ whose cones are generated by the faces of $\Conv(\sigma\cap N\setminus\{0\})$. Moreover, the minimal terminal $\Q$-factorial models of $X_\sigma$ correspond to the subdivisions of $\Sigma_\text{c}$ into simplicial terminal cones whose minimal generators belong to the bounded faces of $\Conv(\sigma\cap N\setminus\{0\})$\cite[pp.552--554]{CLS11}.
\end{rem}

\section{Main results}

In this chapter, we prove the non-existence of moderate toric resolutions for three-dimensional terminal $\Q$-factorial affine toric variety and four-dimensional Gorenstein terminal $\Q$-factorial affine toric variety both of them are special cases of Theorem \ref{d dim no moderate} (Corollary \ref{3Qterminal no moderate}, Corollary \ref{4GorQterminal no moderate}). We also discuss the properties of moderate toric resolutions in dimension three (Theorem \ref{moderate is minQterminal}, Corollary \ref{moderate is crepant}). Furthermore, we give an example of the four-dimensional Gorenstein terminal $\Q$-factorial affine toric varieties which does not admit a Hilbert basis resolution (Theorem \ref{4 no Hilbert basis resolution}), which is a generalization of \cite[Example 3.1]{BGS95}.

\subsection{Properties of moderate toric resolutions for certain normal affine toric varieties in dimensions three and four}

First, the class of toric singularities appears in this section is characterized as follows.
\begin{lem}[Terminal lemma {\cite[pp.34--36]{Oda88}}]\label{terminal lemma}
Let the rank of $N$ be three and let $\sigma\subseteq N_\R$ be a three-dimensional terminal simplicial cone. Then there exists a $\Z$-basis $u_1,u_2,u_3$ of $N$ such that $\sigma=\Cone(u_1,u_2,u_1+pu_2+qu_3)\ (0\leq p<q,\gcd(p,q)=1)$.
\end{lem}
\begin{lem}[\cite{MS84},\cite{BBBK11},{\cite[Proposition 11.4.19 (a)]{CLS11}}]\label{MS84}
Let $G\subseteq\GL(4,\C)$ be an abelian group of order $r$ without pseudo-reflections. Then $\A_\C^4/G$ is a Gorenstein terminal singularity if and only if $G$ is conjugate to a group generated by a diagonal matrix $\diag(\zeta,\zeta^{-1},\zeta^a,\zeta^{-a})$, where $\gcd(a,r)=1$ and $\zeta$ is an $r$-th primitive root of unity.
\end{lem}
\begin{thm}\label{d dim no moderate}
Let $d\geq3$ and let $\sigma\subseteq N_\R$ be a $d$-dimensional cone generated by
\begin{equation*}
e_1=(1,0,\ldots,0),e_2=(0,1,0,\ldots,0),\ldots,e_{d-1}=(0,\ldots,0,1,0),e_d=(a_1,a_2,\ldots,a_d)
\end{equation*}
where $a_1,a_2,\ldots,a_d\in\Z,a_1\leq a_2\leq\cdots\leq a_{d-1}<a_d$ and $1\leq a_1\leq d-2$. Then $X_\sigma$ does not admit a moderate toric resolution.
\end{thm}
\begin{proof}
From Proposition \ref{property of multiplicity}, since $\#P_\sigma=\mult(\sigma)=a_d$, the set
\begin{equation*}
P_\sigma\setminus\{0\}=\{n\in N\mid n=\sum_{i=1}^d\lambda_ie_i,0\leq\lambda_i<1\}\setminus\{0\}
\end{equation*}
consists of the following $a_d-1$ points:
\begin{align*}
p_l&=\sum_{i=1}^{d-1}\left(\left\lceil\frac{a_il}{a_d}\right\rceil-\frac{a_il}{a_d}\right)e_i+\frac{l}{a_d}e_d \\
&=\left(\left\lceil\frac{a_1l}{a_d}\right\rceil,\left\lceil\frac{a_2l}{a_d}\right\rceil,\ldots,\left\lceil\frac{a_dl}{a_d}\right\rceil=l\right)\ (1\leq l\leq a_d-1).
\end{align*}
Here, for a real number $\lambda$, $\lceil\lambda\rceil$ denotes the smallest integer greater than or equal to $\lambda$. Suppose that $\sigma$ admits a moderate toric resolution $\Sigma$. Since $\Sigma$ is a Hilbert basis resolution (Proposition \ref{moderate is Hilbert}) and also since $\tau\coloneqq\Cone(e_1,e_2,\ldots,e_{d-1})$ is smooth and $p_l\notin\tau$, there exists a unique $d$-dimensional cone $\sigma_1$ of $\Sigma$ that contains $\tau$ as a face. This cone is of the form $\Cone(e_1,e_2,\ldots,e_{d-1},p_l),\det(e_1,e_2,\ldots,e_{d-1},p_l)=1$. Since
\begin{align*}
&\det(e_1,e_2,\ldots,e_d)=a_d\neq1, \\
&\det(e_1,e_2,\ldots,e_{d-1},p_l)=l
\end{align*}
for each $1\leq l\leq a_d-1$, respectively, we have $\sigma_1=\Cone(e_1,e_2,\ldots,e_{d-1},p_1)$. Since $p_1=(1,1,\ldots,1)$, the hyperplane spanned by $\Gen(\sigma_1)$ is
\begin{equation*}
H=\{(x_1,\ldots,x_d)\in\R^d\mid\sum_{i=1}^{d-1}x_i-(d-2)x_d=1\}.
\end{equation*}
On the other hand, since
\begin{equation*}
\sum_{i=1}^{d-1}a_i-(d-2)a_d\leq d-2+\sum_{i=2}^{d-1}(a_i-a_d)\leq0,
\end{equation*}
we have $H\cap\R_{\geq0}\cdot e_d=\emptyset$, which contradicts that the subdivision $\Sigma$ of $\sigma$ is a moderate toric resolution. Thus, $\sigma$ does not admit a moderate toric resolution.
\end{proof}
\begin{cor}\label{3Qterminal no moderate}
Let $X_\sigma$ be a three-dimensional terminal $\Q$-factorial singular affine toric variety. Then $X_\sigma$ does not admit a moderate toric resolution.
\end{cor}
\begin{proof}
Since $\sigma$ is singular and from terminal lemma (Lemma \ref{terminal lemma}), we can write $\sigma=\Cone((1,0,0),(0,1,0),(1,p,q))\ (1\leq p\leq q,\gcd(p,q)=1)$. Hence, from Theorem \ref{d dim no moderate}, $X_\sigma$ does not admit a moderate toric resolution.
\end{proof}
\begin{cor}\label{4GorQterminal no moderate}
Let $X_\sigma$ be a four-dimensional Gorenstein terminal $\Q$-factorial singular affine toric variety. Then $X_\sigma$ does not admit a moderate toric resolution.
\end{cor}
\begin{proof}
Since $\sigma$ is singular and from Lemma \ref{MS84}, $\sigma$ is generated by $(1,0,0,0)$, $(0,1,0,0)$, $(0,0,1,0)$, $(1,a,r-a,r)$ $(1\leq a\leq r-a<r,\gcd(a,r)=1)$ \cite[Section 5.1]{CLS11}. Hence, from Theorem \ref{d dim no moderate}, $X_\sigma$ does not admit a moderate toric resolution. 
\end{proof}
\begin{dfn}[{\cite[Definition 2.19]{BGS95}}]
Let $\Sigma$ be a fan in $N_\R$ and let $\sigma\subseteq N_\R$ be a cone. By $\Sigma_{|\sigma}$, we denote the set of cones of $\Sigma$ contained in $\sigma$.
\end{dfn}
\begin{rem}[{\cite[Remark 2.20]{BGS95}}]
The set $\Sigma_{|\sigma}$ is a fan. It is a subdivision of $\sigma$ if and only if its support $\bigcup_{\tau\in\Sigma_{|\sigma}}$ is $\sigma$.
\end{rem}
\begin{lem}[{\cite[Theorem 2.22]{BGS95}}]\label{Hilbert dominate minQterminal}
Any Hilbert basis resolution of three dimensional affine toric variety $X$ dominates a minimal terminal $\Q$-factorial model of $X$.
\end{lem}
\begin{thm}\label{moderate is minQterminal}
Let $X_\sigma$ be a three-dimensional affine toric variety. If $X_\sigma$ admits a moderate toric resolution, then it is a minimal terminal $\Q$-factorial model of $X_\sigma$.
\end{thm}
\begin{proof}
Let $X_\Sigma$ be a moderate toric resolution of $X_\sigma$ with the corresponding fan $\Sigma$. Since a moderate toric resolution is a Hilbert basis resolution (Proposition \ref{moderate is Hilbert}), there exists a minimal terminal simplicial model $\Sigma_\text{t}$ of $\sigma$ which is dominated by $\Sigma$ (Lemma \ref{Hilbert dominate minQterminal}). Now, assuming that $\Sigma_\text{t}$ is singular, we can take a singular three-dimensional cone $\sigma_0$ of $\Sigma_\text{t}$. Since $\Sigma$ is a resolution of $\Sigma_\text{t}$, $\Sigma_{|\sigma_0}$ is a resolution of $\sigma_0$. Also, for any nondegenerate cone $\tau\in\Sigma_{|\sigma_0}\subseteq\Sigma$, since $\Sigma$ is a moderate toric resolution, the affine hyperplane spanned by $\Gen(\tau)$ intersects with every one-dimensional face of $\sigma$. Therefore, it also intersects with every one-dimensional face of $\sigma_0$ which contained in $\sigma$. It means that the $\Sigma_{|\sigma_0}$ is a moderate toric resolution of $\sigma_0$, but this contradicts Corollary \ref{3Qterminal no moderate}, because $\sigma_0$ is a three-dimensional terminal simplicial singular cone. Hence, $\Sigma_\text{t}$ is smooth and $\Gen(\Sigma)=\Gen(\Sigma_\text{t})$. Then $\Sigma=\Sigma_\text{t}$ since $\Sigma$ is a subdivision of simplicial fan $\Sigma_\text{t}$.
\end{proof}
\begin{rem}
In any dimension, if a minimal terminal $\Q$-factorial model of $X$ is smooth, then it is a moderate toric resolution of $X$ \cite[Proposition 6.10]{CMDY24}.
\end{rem}
\begin{cor}\label{moderate is crepant}
Let $X_\sigma$ be a three-dimensional canonical affine toric variety. If $X_\sigma$ admits a moderate toric resolution, then it is a crepant resolution of $X_\sigma$.
\end{cor}
\begin{proof}
As mentioned in Remark \ref{canonical terminal}, the minimal terminal simplicial models of $\sigma$ correspond to the subdivisions of canonical model $\Sigma_\text{c}$ of $\sigma$ into simplicial terminal cones whose minimal generators belong to the affine hyperplane $H$ spanned by $\Gen(\sigma)$. Hence, from Theorem \ref{moderate is minQterminal}, the minimal generators of every one-dimensional cone appearing in a moderate toric resolution of $\sigma$ are contained in $H$. Therefore, it is a crepant resolution of $\sigma$.
\end{proof}

\subsection{Four-dimensional Gorenstein terminal $\Q$-factorial affine toric varieties which do not admit a Hilbert basis resolution}

As an analogue of Corollary \ref{4GorQterminal no moderate}, an example of a four-dimensional Gorenstein terminal $\Q$-factorial affine toric variety which does not admit a Hilbert basis resolution was given in \cite[Example 3.1]{BGS95}. In this section, we present the following theorem as a generalization of this example.
\begin{thm}\label{4 no Hilbert basis resolution}
Let $X_\sigma$ be a four-dimensional Gorenstein terminal $\Q$-factorial affine toric variety associated to the cone $\sigma$ generated by
\begin{equation*}
e_1=(1,0,0,0),e_2=(0,1,0,0),e_3=(0,0,1,0),e_4=(1,a,r-a,r)
\end{equation*}
where $\gcd(a,r)=1,1<a<r-a<r$ and $2\nmid r$ (Lemma \ref{MS84}, \cite[Section 5.1]{CLS11}). Then $X_\sigma$ does not admit a Hilbert basis resolution.
\end{thm}
To prove this theorem, we first show a lemma. We keep the notation in Theorem \ref{4 no Hilbert basis resolution}.

Let $P_\sigma=\{n\in N\mid n=\sum_{i=1}^4\lambda_ie_i,0\leq\lambda_i<1\}$. Then from Proposition \ref{property of multiplicity}, since $\#P_\sigma=\mult(\sigma)=r$, the set
\begin{equation*}
P_\sigma\setminus\{0\}=\{n\in N\mid n=\sum_{i=1}^4\lambda_ie_i,0\leq\lambda_i<1\}\setminus\{0\}
\end{equation*}
consists of the following $r-1$ points:
\begin{align*}
p_l&=\left(\left\lceil\frac{l}{r}\right\rceil-\frac{l}{r}\right)e_1+\left(\left\lceil\frac{al}{r}\right\rceil-\frac{al}{r}\right)e_2+\left(\left\lceil\frac{(r-a)l}{r}\right\rceil-\frac{(r-a)l}{r}\right)e_3+\frac{l}{r}e_4 \\
&=\left(\left\lceil\frac{l}{r}\right\rceil,\left\lceil\frac{al}{r}\right\rceil,\left\lceil\frac{(r-a)l}{r}\right\rceil,l\right)\ (1\leq l\leq r-1).
\end{align*}
\begin{lem}\label{determinant}
For $1\leq l,l'\leq r-1$,
\begin{align}
&\det(e_1,e_2,e_3,p_l)=l, \notag \\
&\det(p_l,e_2,e_3,p_{l'})=l'-l, \notag \\
&\det(p_l,e_2,e_3,e_4)=r-l, \notag \\
&\det(e_1,p_l,e_3,e_4)=r\left\lceil\frac{al}{r}\right\rceil-al=r+(r-a)l-r\left\lceil\frac{(r-a)l}{r}\right\rceil, \notag \\
\label{nontrivial}
&\det(e_1,p_l,p_{l'},e_4)=r\left(\left\lceil\frac{(r-a)l'}{r}\right\rceil-\left\lceil\frac{(r-a)l}{r}\right\rceil\right)-(r-a)(l'-l), \tag{$\ast$} \\
&\det(e_1,e_2,p_l,e_4)=r\left\lceil\frac{(r-a)l}{r}\right\rceil-(r-a)l \notag.
\end{align}
\end{lem}
\begin{proof}
We only show \eqref{nontrivial}, since the other equalities are straightforward. Noting that $\gcd(a,r)=1$, we get
\begin{align*}
\det(e_1,p_l,p_{l'},e_4)&=
\begin{vmatrix}
\left\lceil\frac{al}{r}\right\rceil & \left\lceil\frac{(r-a)l}{r}\right\rceil & l \\
\left\lceil\frac{al'}{r}\right\rceil & \left\lceil\frac{(r-a)l'}{r}\right\rceil & l' \\
a & r-a & r
\end{vmatrix} \\
&=\left\lceil\frac{al}{r}\right\rceil\left(l'+\left\lceil\frac{-al'}{r}\right\rceil\right)r+\left\lceil\frac{(r-a)l}{r}\right\rceil l'a+\left\lceil\frac{al'}{r}\right\rceil(rl-al) \\
&\quad -\left\lceil\frac{(r-a)l'}{r}\right\rceil al-\left\lceil\frac{al}{r}\right\rceil(rl'-al')-\left(l+\left\lceil\frac{-al}{r}\right\rceil\right)\left\lceil\frac{al'}{r}\right\rceil r\\
&=\left\lceil\frac{al}{r}\right\rceil\left(1-\left\lceil\frac{al'}{r}\right\rceil\right)r+al'(l+1)-al(1+l')-\left(1-\left\lceil\frac{al}{r}\right\rceil\right)\left\lceil\frac{al'}{r}\right\rceil r \\
&=r\left(\left\lceil\frac{al}{r}\right\rceil-\left\lceil\frac{al'}{r}\right\rceil\right)+a(l'-l) \\
&=r\left(\left(l+1-\left\lceil\frac{(r-a)l}{r}\right\rceil\right)-\left(l'+1-\left\lceil\frac{(r-a)l'}{r}\right\rceil\right)\right)+a(l'-l) \\
&=r\left(\left\lceil\frac{(r-a)l'}{r}\right\rceil-\left\lceil\frac{(r-a)l}{r}\right\rceil\right)-(r-a)(l'-l).
\end{align*}
\end{proof}
\begin{proof}[Proof of Theorem \ref{4 no Hilbert basis resolution}]
We follow arguments in \cite[Example 3.1 ($a=3,r=7$)]{BGS95}. For each $1\leq i\leq r-1$, there exists a unique $1\leq l_i\leq r-1$ such that $l_i\equiv bi\mod r$. Here $b$ is an integer which satisfies $ab\equiv1\mod r,\ 1\leq b\leq r-1$. For this integer $l_i$, we have $-(r-a)l_i\equiv i\mod r$, and hence
\begin{equation*}
r\left\lceil\frac{(r-a)l_i}{r}\right\rceil-(r-a)l_i=i.
\end{equation*}
Assume that $\sigma$ admits a Hilbert basis resolution $\Sigma$. Since $\tau\coloneqq\Cone(e_1,e_2,e_4)$ is smooth, there exists a four-dimensional cone of $\Sigma$ that contains $\tau$ as a face. A candidate for such a cone is $\Cone(e_1,e_2,p_l,e_4)$ satisfying $\det(e_1,e_2,p_l,e_4)=1$. From Lemma \ref{determinant}, since $\det(e_1,e_2,p_l,e_4)=r\lceil(r-a)l/r\rceil-(r-a)l$, we have $l=l_1$. Thus, $\Sigma$ contains two four-dimensional cones with $\Cone(e_1,e_4,p_{l_1})$ as their common face. The cone $\Cone(e_1,e_2,p_{l_1},e_4)$ is one of them, and the other is either $\Cone(e_1,p_{l_1},e_3,e_4)$ or $\Cone(e_1,p_{l_1},p_l,e_4)$. Since $\det(e_1,p_{l_1},e_3,e_4)=r-1$, this cone is $\Cone(e_1,p_{l_1},p_l,e_4)$ satisfying $\det(e_1,p_{l_1},p_l,e_4)=1$. Again from Lemma \ref{determinant},
\begin{align*}
\det(e_1,p_{l_1},p_l,e_4)&=r\left(\left\lceil\frac{(r-a)l}{r}\right\rceil-\left\lceil\frac{(r-a)l_1}{r}\right\rceil\right)-(r-a)(l-l_1) \\
&=-1+r\left\lceil\frac{(r-a)l}{r}\right\rceil-(r-a)l,
\end{align*}
so $l=l_2$. Consequently, $\Sigma$ contains $\Cone(e_1,p_{l_1},p_{l_2},e_4)$.

Similarly, if $\Sigma$ contains $\Cone(e_1,p_{l_{i-1}},p_{l_i},e_4)\ (2\leq i\leq r-2)$. Then there exists two four-dimensional cones of $\Sigma$ that contain $\Cone(e_1,e_4,p_{l_i})$ as their common face. The cone $\Cone(e_1,p_{l_{i-1}},p_{l_i},e_4)$ is one of them, and the other is either $\Cone(e_1,e_2,p_{l_i},e_4)$, $\Cone(e_1,p_{l_i},e_3,e_4)$ or $\Cone(e_1,p_{l_i},p_l,e_4)$. Since $\det(e_1,e_2,p_{l_i},e_4)=i\neq1$ and $\det(e_1,p_{l_i},e_3,e_4)=r-i\neq1$, this cone is $\Cone(e_1,p_{l_i},p_l,e_4)$ satisfying $\det(e_1,p_{l_i},p_l,e_4)=1$. Now from Lemma \ref{determinant},
\begin{equation*}
\det(e_1,p_{l_i},p_l,e_4)=-i+r\left\lceil\frac{(r-a)l}{r}\right\rceil-(r-a)l,
\end{equation*}
so $l=l_{i+1}$. Consequently, $\Sigma$ contains $\Cone(e_1,p_{l_i},p_{l_{i+1}},e_4)$. Also by repeating this operation, we see that $\Sigma$ contains $\Cone(e_1,p_{l_{r-1}},e_3,e_4)$. In particular, $\Sigma$ contains $\Cone(p_{l_{(r-1)/2}},p_{l_{(r+1)/2}})$ since it is a face of $\Cone(e_1,p_{l_{(r-1)/2}},p_{l_{(r+1)/2}},e_4)$.

If we apply the above argument for $\Cone(e_1,e_2,e_3)$, we see that $\Sigma$ contains $\Cone(e_1,e_2,e_3,p_1),\Cone(p_l,e_2,e_3,p_{l+1})\ (1\leq l\leq r-2)$ and $\Cone(p_{r-1},e_2,e_3,e_4)$. In particular, $\Sigma$ contains $\Cone(p_{(r-1)/2},p_{(r+1)/2})$.

Since $b\neq1,r-1$, $\Cone(p_{l_{(r-1)/2}},p_{l_{(r+1)/2}})\neq\Cone(p_{(r-1)/2},p_{(r+1)/2})$. We see that the intersection of $\Cone(p_{l_{(r-1)/2}},p_{l_{(r+1)/2}})$ and $\Cone(p_{(r-1)/2},p_{(r+1)/2})$ is a one-dimensional cone $\R_{\geq0}\cdot e$, where
\begin{equation*}
e=p_{l_{r-1/2}}+p_{l_{r+1/2}}=p_{r-1/2}+p_{r-1/2}=(2,a+1,r-a+1,r)=\sum_{j=1}^4e_j.
\end{equation*}
But this cone is not a face of $\Cone(p_{l_{(r-1)/2}},p_{l_{(r+1)/2}})$ and $\Cone(p_{(r-1)/2},p_{(r+1)/2})$. Therefore, $\sigma$ does not admit a moderate toric resolution.
\end{proof}
\begin{rem}
There exists a four-dimensional Gorenstein terminal $\Q$-factorial singular affine toric variety which admits a Hilbert basis resolution. For example, in the notation of Theorem \ref{4 no Hilbert basis resolution}, when $a=1$ and $r=3$, we can compute a toric resolution of $X_\sigma$ by using Macaulay2 \cite{M2} as follows:
\begin{verbatim}
i1 : loadPackage"NormalToricVarieties"

o1 = NormalToricVarieties

o1 : Package

i2 : rayList={{1,0,0,0},{0,1,0,0},{0,0,1,0},{1,1,2,3}};coneList={{0,1,2,3}};

i4 : X=normalToricVariety(rayList,coneList);Y=makeSmooth X;rays Y,max Y

o6 = ({{1, 0, 0, 0}, {0, 1, 0, 0}, {0, 0, 1, 0},

     {1, 1, 2, 3}, {1, 1, 1, 1}, {1, 1, 2, 2}},

     {{0, 1, 2, 4}, {0, 1, 3, 4},{0, 2, 3, 5}, {0, 2, 4, 5},

     {0, 3, 4, 5}, {1, 2, 3, 5}, {1, 2, 4, 5}, {1, 3, 4, 5}})

o6 : Sequence
\end{verbatim}
The obtained toric resolution correspond to the fan consisting of four-dimensional cones
\begin{align*}
&\Cone(e_1,e_2,e_3,p_1),\Cone(e_1,e_2,e_4,p_1),\Cone(e_1,e_3,e_4,p_1),\Cone(e_1,e_3,p_1,p_2) \\
&\Cone(e_1,e_4,p_1,p_2),\Cone(e_2,e_3,e_4,p_2),\Cone(e_2,e_3,p_1,p_2),\Cone(e_2,e_4,p_1,p_2)
\end{align*}
and their faces. On the other hand, the Hilbert basis of $\sigma$ consists of the generators $e_1,e_2,e_3,e_4$ of $\sigma$ and $p_1=(1,1,1,1),p_2=(1,1,2,2)$. Thus, the above toric resolution happens to be a Hilbert basis resolution of $\sigma$.
\end{rem}

\bibliographystyle{alpha}
\bibliography{reference}

\end{document}